\documentclass[11pt]{article}
\usepackage{amsmath, amssymb, amsfonts, amsthm}
\usepackage{enumitem, multicol}
\usepackage{color}
\usepackage{graphicx,caption}
\usepackage{latexsym,booktabs}
\usepackage{amsthm}
\usepackage{hyperref}
\usepackage{tikz}
\usepackage{caption}
\usepackage{multicol}
\usepackage{multirow}
\usepackage{romannum}
\usepackage{hyperref}
\usepackage{url}
\allowdisplaybreaks
\def \N {{\mathbb N}}
\def \Z {{\mathbb Z}}
\def \R {{\mathbb R}}

\def \0 {{\mathbf 0}}

\newtheorem{theorem}{Theorem}[section]
\newtheorem{lemma}[theorem]{Lemma}
\newtheorem{cor}[theorem]{Corollary}
\newtheorem{defn}{Definition}[section]
\newtheorem{ex}[theorem]{Example} %[section]

\newtheorem{rem}[theorem]{Remark}

\def\[{[\hskip-1pt [}
\def\]{]\hskip-1pt ]}

\def\R{{\mathbb R}}

\def\N{{\mathbb N}}

\def\++{\boxplus}

\title{A Note on the Exponents of Primitive Companion Matrices}
\begin{document}
\baselineskip 15pt
\author {Monimala Nej, A. Satyanarayana Reddy\\
Department of 
Mathematics, Shiv Nadar 
University, India-201314\\ (e-mail: 
mn636@snu.edu.in, satyanarayana.reddy@snu.edu.in).
  }
\date{}   
\maketitle
\begin{abstract}
 A nonnegative matrix $A$ is said to be {\it primitive} if for some positive integer $m$, entries in $A^m$ are positive, notationally represented as $A^m>0.$ The smallest such $m$ is called the {\it exponent} of $A$, denoted $exp(A).$ For the class of primitive companion matrices $X$, we find $exp(A)$ for certain $A \in X.$ Thereafter, we find certain numbers in $E_n(X)$, where $E_n(X)=\{m \in \N : \text{there  exists an  $n \times n$ matrix $A$ in $X$ with} \; exp(A)=m\}$. At the end we propose open problems for further research. 
\end{abstract}
{\bf{AMS(2010)}}: 05C50, 05C38,15B99.
\section{Introduction} 
Let $M_n(\R)$ be the set of all real matrices of order $n$. We will denote the $ij$-th entry in $A \in M_n(\R)$  by $a_{ij}$ or $A_{ij}.$  If $A,\;B \in M_n(\R),$ then by $A\ge B$ ($A>B$) we mean 
$a_{ij}\ge b_{ij}$ ($a_{ij}>b_{ij}$) for all $i,j.$ In particular, if $B$ is the zero matrix, then $A$ is said to be a {\it nonnegative (positive)} matrix, denoted by $A \geq 0$ ($A >0$).
An {\it irreducible} matrix which is not primitive is called an {\it imprimitive} matrix. We refer to Henryk Minc \cite{minc} for the definition and properties of an irreducible matrix. For a class $X$ of nonnegative matrices, the {\it exponent set} $E_n(X)$ of $X$ is defined as:
$$E_n(X)=\{m\in \N : \mbox{ there exists a primitive matrix $A$ of order $n$ in X with 
$exp(A)=m$}\}.$$ 
It is very easy to observe that $1\in E(M_n(\R)).$ In 1950, Helmut Wielandt proved that if $A\in M_n(\R)$ is a primitive matrix, then $exp(A)\leq \omega_{n}$, where 
$\omega_{n}=(n-1)^2+1.$ This bound is well known as the Wielandt bound.  
For the proof, we refer to Hans Schneider \cite{HS}, Holladay and 
Varga \cite{Ho:Va} and Perkins \cite{Per}. Wielandt also proved that $\omega_n$ is a sharp bound, that is, there exists a primitive matrix of order $n$ whose exponent is $\omega_n.$ As a consequence, $E(M_n(\R))\subseteq [1,\omega_n],$ where $[a,b]$ denotes the set $\{i\in \Z  : a\le i\le b\}.$
In 1964, A.L. Dulmage and N.S. Mendelsohn \cite{D:M} found that $E(M_n(\R))\subset[1,\omega_n].$ Research focused on primitive exponents ever since 1950, when Wielandt published his paper \cite{Wie}. For a given class $X$ of nonnegative matrices, finding $E_n(X),$ bounds on $E_n(X)$ and matrices in $X$ which attain those bounds are the major parts of the literature. For instance, the papers~\cite{Bru:Ross}, \cite{Lewin}, 
\cite{Liu1}, \cite{Shao}, \cite{B:M:W},\cite{Ross}, \cite{Dul:Men}  studied those problems for 
different classes of nonnegative matrices.

Let $C_n$ be the  set of all $(0,1)$
companion matrices of polynomials  of the form 
$x^{n}-\sum\limits_{i=0}^{n-1}a_{i}x^{i}$, where $a_i\in \{0,1\}.$ That is, $C_n=\{A\in M_n(\R) : (a_{i\;i+1})_{1\leq i \leq n-1}=1,  (a_{ni})_{1 \leq i \leq n} =a_{i-1} \;
{\text{and}} \; a_{ij}=0 \; {\text {otherwise}}\}.$ In this paper we wish to investigate some problems on primitive companion matrices. The set of all primitive $(0,1)$ companion matrices will be denoted by $CP_n.$ As per the best of our knowledge no one has studied the number of imprimitive matrices and the number of primitive matrices with a given 
exponent in a given class. Furthermore, there is no specific formula for 
 computing the exponent of a given matrix from a given class.  Here in this work, we are interested in solving these problems for $C_n.$

It is easy to 
verify that a nonnegative matrix is primitive if and only if its sign matrix is primitive. If $A$ is a nonnegative matrix, the {\it sign matrix} of $A,$ denoted $\it{sign}(A),$ is the $(0,1)$ matrix such that $\it{sign}(A)_{ij}=1$ if and only if $a_{ij}>0.$ Furthermore, the exponent of a primitive matrix is always the same as the exponent of its sign matrix. It is thus sufficient to work with primitive $(0,1)$ companion matrices. Also, in the context of powers of matrices, this means that the algebra of interest is the Boolean algebra and thus $1+1=1$ naturally follows. Since the first 
$n-1$ rows of every matrix in $C_n$ are fixed, it is sufficient to 
specify 
the last row. Clearly, 
there is a bijection between $C_n$ and $B_n$, where $B_n$ denotes the set 
of 
all binary strings of length $n.$ In particular, 
$|C_n|=|B_n|=2^n,$ where $|S|$ denotes the number of elements in the set $S.$  Now the elements in $C_n$ will be denoted as $A_Y$, where $Y \in B_{n-1}$ and the
last row of $A_Y$ will be $1Y$ or $0Y$ accordingly as $A_Y$ is irreducible or reducible respectively.  
We arrange the contents of this paper as follows. In section \ref{sec:loop}, we find the number of primitive $(0,1)$ companion matrices of order $n.$ Also we find the exponent of $A,$ where $A \in CP_n$ and the trace of $A$ is positive. In Section \ref{sec:nloop}, we discuss the exponents of primitive $(0,1)$ companion matrices with zero trace. At the end, we show the existence of certain numbers in $E_n(CP_n)$. It should be noted that the non-existence of some numbers in $E(CP_n)$ follows from A.L. Dulmage and N.S. Mendelsohn \cite{D:M}, M. Lewin and Y. Vitek \cite{M:Y} and Ke Min Zhang \cite{ZKM}. Finally, we suggest some open problems in Section \ref{sec:prob}. In the rest of this section we will discuss a few preliminaries and notations required for the rest of the paper. \\

We denote $D=(V,E)$ as a digraph (directed graph) with the vertex set $V=V(D),$ the edge set $E=E(D)$ and order $n=|V|.$ Throughout the paper, loops are permitted but no multiple edges. A $u \to v$ {\it walk} in $D$ is a sequence of vertices $u, u_1, u_2, \ldots, u_l=v$ and a sequence of edges $(u,u_1), (u_1,u_2), \ldots, (u_{l-1},v)$ where vertices and edges may be repeated. A {\it cycle or closed walk} is a $u \to v$ walk where $u=v.$ A {\it path} is a walk with distinct vertices. An {\it elementary cycle} is a closed $u \to v$ walk with distinct vertices except for $u=v.$ The {\it length of a walk} is the number of edges in the walk. The notation $u \xrightarrow{l} v$ (resp. $u \not\xrightarrow{l} v$) is used to indicate that there is a $u \to v$ walk (resp. no $u \to v$ walk) of length $l$ and the notation $u \xrightarrow{l^+} v$ to indicate that $u \xrightarrow{m} v$ for all $m \geq l.$

For an $n \times n$ $(0,1)$ matrix $A,$ the {\it adjacency digraph}, denoted by $D(A)$, is the digraph $D=(V,E)$ such that $V=\{1,2,\ldots,n\}$ and $(i,j) \in E$ if and only if $a_{ij}=1.$ On the other hand, for a digraph $D$, the {\it adjacency matrix} $A$ of $D$ is defined as follows: $a_{ij}=1$ if there is an edge from $i$ to $j$ in $D$ and $a_{ij}=0$ otherwise. If $D$ is the adjacency digraph of $A$, then $D^k$ is defined to be the adjacency digraph of $A^k.$ It is easy to observe that the $ij$-th entry of $A^k$ is $1$ if and only if there is a walk of length $k$ from $i$ to $j$ in $D.$
 %%%%%%%%%%
 If $D$ is a digraph, then  $D$ is primitive (imprimitive) 
if and only if its adjacency matrix is  primitive (imprimitive)  and $exp(D)$ is defined to 
be  the exponent of its adjacency matrix. Hence one can interchangeably use a  matrix or its adjacency digraph for the purpose of establishing the primitivity and exponent of the matrix. More detail about digraphs and primitivity of digraphs is available in Richard A. Brualdi and Herbert J. Ryser \cite{bru} and Henryk Minc \cite{minc}.

For $A\in C_n$ we define $V_1(A), V_2(A)$ (or simply $V_1,V_2$ if 
$A$ is clear from the context)  as 
$$V_1=\{i \in [1,n] : a_{ni}=0\}\;\; and\;\; V_2=\{i \in [1,n] : 
a_{ni}>0\}.$$ 
For  $U\subseteq [1,n]$ write  $U=U_1\cup U_2\cup \dots \cup U_r$, where for 
each $ k\in [1,r]$, $U_k=[i_k, j_k]$ with $1 \leq i_k \leq j_k \leq n,$ and for each $k \in [2,r],$ $j_{k-1}+2 \leq i_k.$ We define 
$m(U)=\max\{|U_1|, 
|U_2|,\ldots, |U_r|\}.$  For $U= \emptyset$, $m(U)$ is assumed to be zero.  
For example, if $U=\{2,3,5,7,8,9,10,13,14\},$ then $U=\{2,3\}\cup \{5\}\cup 
\{7,8,9,10\}\cup \{13,14\}$ and 
$m(U)=|\{7,8,9,10\}|=4.$ Note that $m(V_1)\in [0, n-2]$ whenever $A \in CP_n.$ 

\begin{ex}\label{ex:1011}
From the digraph $D$ in Figure \ref{fig:figure1}, we have  $V_1=\{2\}\cup \{4\}\cup 
\{6\}\cup \{8, 9, 10\},$ $m(V_1)=3$ and $V_2=\{1,3,5,7\}.$ 
\end{ex}
\begin{figure}[hbt]
\begin{center}
\includegraphics[width=.2\linewidth, width=.2\textwidth]{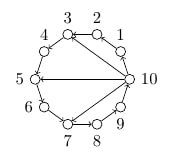}
\end{center}
\caption{}  \label{fig:figure1}
\end{figure}
%%%%%%%%%%%%%%%%%%%%%%%%%%%%%%%%%%%%%%%%%%%%%%%%%
\section{Number of primitive companion matrices and exponent of $A \in CP_n$ with a positive trace } \label{sec:loop}
%%%%%%%%%%%%%%%%%%%%%%%%%%%%%%%%%%%%%%%%%%%%%%%%%
It is known that if $A$ is a primitive matrix, then $A$ is irreducible, but the converse is not true. For example, consider $A \in C_n$ with $a_{n0}=1$ and $a_{ni}=0, 1 \leq i \leq n-1.$ Then $A$ is irreducible but not primitive. Moreover, the primitive and imprimitive matrices are defined over the class of irreducible matrices. We let $C'_n$ denote the set of all irreducible matrices in $C_n$. Then $|C'_n|=2^{n-1}$ follows from the following definition.
\begin{defn}(\cite{minc})
\begin{enumerate}
\item
A digraph $D$ is said to be {\it strongly connected} if there exist a walk from $u$ to $v$ for each $u$ and $v \in V.$ 
\item
Let $D$ be a strongly connected digraph. The greatest common divisor $h$ of the lengths of all elementary cycles in $D$ is called the {\it index of imprimitivity} of the digraph $D.$ If $h=1$, then $D$ is called a primitive digraph; otherwise, it is called imprimitive.
\end{enumerate}
\end{defn}
It is easy to see that a nonnegative matrix $A$ is irreducible if and only if $D(\it{sign}(A))$ is strongly connected. Now the following theorem gives the number of imprimitive matrices in $C'_n$. 
 \begin{theorem}
Let $n \geq 3$ be an integer. Then $$|C'_n \setminus CP_n| = \sum\limits_{k=1}^{r}(-1)^{k+1}\left(\sum\limits_{1 \leq i_1<i_2 < \dots < i_k \leq r} 2^{\frac{n}{p_{i_1}p_{i_2} \dots p_{i_k}}-1}\right),$$  where $p_{1},p_{2},\ldots ,p_{r}$ are all possible distinct prime factors of $n$.
\end{theorem}
\begin{proof}
Let us denote $A_i=\{y \cdot p_i : y \in [1,(\frac{n}{p_i}-1)]\}$ and $B_i=\{X \subseteq A_i : X \neq \emptyset \}$ for each $i \in [1,r].$  Then it is easy to see that $A \in C'_n$ is imprimitive if the length of every elementary cycle, excluding the length $n$, in $D(A)$ belongs to $B_i$ for some $i.$ Hence for each $i,$ there are $|B_i|$ imprimitive matrices in $C'_n.$ Furthermore, if $D(A)$ contains exactly one cycle, then the length of the cycle is $n,$ and $A$ is imprimitive. Now $|B_{i}|=(2^{\frac{n}{p_{i}}-1}-1)$ and by inclusion-exclusion principle 
\begin{align}
|C'_n \setminus CP_n|&=\sum\limits_{1 \leq i_1 \leq r} 2^{\frac{n}{p_{i_1}}-1}-\sum\limits_{1\leq i_1 < i_2 \leq r} 2^{\frac{n}{p_{i_1}p_{i_2}}-1}+\sum\limits_{1\leq i_1 < i_2 < i_3\leq r} 2^{\frac{n}{p_{i_1}p_{i_2}p_{i_3}}-1}-\dots (-1)^{r+1}2^{\frac{n}{p_{1}p_{2}\dots p_{r}}-1} \nonumber\\
                      &=\sum\limits_{k=1}^{r}(-1)^{k+1}\left(\sum\limits_{1 \leq i_1<i_2 < \dots < i_k \leq r} 2^{\frac{n}{p_{i_1}p_{i_2} \dots p_{i_k}}-1}\right).\qedhere                                                                                                                                                              
                      \end{align} 
                    
\end{proof}

For an example, let us take $n=8.$ Then the number of imprimitive matrices of order 8 is $2^{(\frac{8}{2}-1)} = 8$ and they are given by $A_Y$, where $$Y \in \{0000000,0100000,0001000,0000010,0101000,0100010,0001010,0101010\}.$$  Hence $|CP_8|=(2^7-8)=120.$ For $n=10$, there are $2^4+2^1-1=17$ imprimitive matrices in $C'_{10}.$

From now onwards, we confine our study to primitive companion  matrices and focus on finding $E(CP_n).$ For this purpose it is sufficient to know $exp(A)$ whenever $A \in CP_n, n \geq 2.$ In this paper we find the $exp(A)$ for certain cases and we provide some results which may be helpful in finding the $exp(A)$ for the remaining cases. We now recall a general procedure that evaluates the exponent of a primitive matrix.

Let $A$ be a primitive matrix. Then
\begin{enumerate}
 \item $exp(A:i,j)$ is defined to be the 
smallest positive integer $k$ such that $i \xrightarrow{k^+} j$ in $D(A).$ Equivalently, 
$A^l_{ij}>0$ for any integer $l \geq k.$
\item $exp(A:i)$ is the smallest positive integer 
$p$ such that $i \xrightarrow{p} j$ for all $j \in V.$  Equivalently, every entry in the $i^{th}$ row of $A^p$ is 
positive. As a consequence, every entry in 
the $i^{th}$ row of $A^{p+1}$ is also  positive.
\end{enumerate}
The following result is well known. For instance,  see \cite{bru}.
\begin{lemma}\label{lemma:exp(A)}
 Let $A\in M_n(\R)$ be a primitive matrix. Then
 $$exp(A)=\max\limits_{1 \leq i,j \leq n}exp(A:i,j)=\max\limits_{1 \leq i \leq 
n}exp(A:i).$$
\end{lemma}
\begin{ex} Consider the matrix $A$ with $D(A)$ given in Figure \ref{fig:figure2}. Note that there is an edge $(i,j)$ in $D(A)$ if and only if $(j,i)\in E.$ It is easy to check that $exp(A:8,1)=1$, $exp(A:8,8)=2$ and $exp(A:i)=2$ for all $i.$ Hence $exp(A)=2.$
\end{ex}  
\begin{figure}[hbt]
\begin{center}
\includegraphics[width=.2\linewidth, width=.2\textwidth]{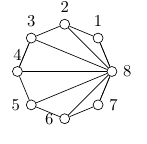}
\end{center}
\caption{} \label{fig:figure2}
\end{figure}
Thus, our goal is to find $exp(A:i,j)$ for all $i,j\in [1,n].$ Before proceeding to the next result, we need the following notation. We denote $L=\{\ell_{i} : 1 \leq i \leq z\}$, where $\ell_{1},\ell_{2},\dots,\ell_{z}$ are all possible distinct elementary cycle lengths in $D(A)$ with $\ell_{1}<\ell_{2}<\dots<\ell_{z}.$  Let $B_n^{x,k}\subseteq B_n$ denote the set of all binary strings with $x$ zeros and having at least 
one longest sub-word of zeros of length $k.$ For example, 
$B_6^{4,2}=\{100100,010100,010010,001100,001010,001001\}.$ 
Consequently, a  necessary condition for  $B_n^{x,k}$ to be nonempty is $n\ge 
x\ge 
k\ge 0.$  An immediate observation is that $B_n=\cup_{x=0}^n\cup_{k=0}^x 
B_n^{x,k}.$ Thus $2^n=\sum\limits_{x=0}^n\sum\limits_{k=0}^x F_n(x,k),$
where $F_n(x,k)=|B_n^{x,k}|.$ The value of $F_n(x,k)$ is defined to be zero 
whenever $n< 0.$ For the basic results and facts about $F_n(x,k),$ we refer the reader to Monimala Nej and A. 
Satyanarayana Reddy \cite{M:S}. M.A. Nyblom \cite{NY1} 
denoted $S_{r}(n)$ for the set of all binary strings 
of length $n$ without  $r$-runs of ones, where $n\in \mathbb{N}$ and $r \geq 
2$, and $T_{r}(n)=|S_{r}(n)|$. For example, if  $n=3$ and $r=2$, then 
$S_{2}(3)=\{000, 101,001,100,010\}$ and $T_{2}(3)=5$.
\begin{theorem}\label{positive trace}
Let $A \in CP_n$ such that the trace of $A$ is positive. Then $exp(A)=n+m(V_1)$.
\end{theorem}
\begin{proof}
Suppose that $i \in V$  and $j \in V_2$.  Then $exp(A:i,j)=n-i+1$. If $j \in V_1$, then $exp(A:i,j)=n-i+1+p$, where $p=j-\max\{[1,j] \cap V_2\}$ and $1 \leq p \leq j-1.$ Thus for each $i \in V,$ it follows from Lemma \ref{lemma:exp(A)} that  
$$exp(A:i)=\max\limits_{1 \leq j \leq n} \; \{exp(A:i,j)\}=n-i+1+m(V_1).$$ Hence 

 $$exp(A)  =  \max_{1 \leq i \leq n} \; \{exp(A:i)\} = \max_{1 \leq i \leq n}(n-i+1+m(V_1))  = n+m(V_1).$$
\end{proof}
\begin{cor}
\begin{enumerate}
\item
For a given $n,$ $m(V_1) \in [0, n-2].$ Hence $[n, 2(n-1)] \subset E(CP_n).$
\item 
For $t \in [n,2(n-1)]$, the number of matrices with a positive trace and with exponent $t$ in $CP_{n}$ is given by $\sum\limits_{x=k}^{n-2}F_{n-2}(x,k)$, where $k=t-n$ and $F_{n}(x,k)$ has been  described in \cite{M:S}.
\end{enumerate}
\end{cor}

\begin{ex}
The exponent of the digraph in Figure \ref{fig:figure3} is $11.$ Observe that $1 \not\xrightarrow{10} 4.$
%%%%%%%%%%%%%%%%%%%%%%%%%

\begin{figure}[htb]
\begin{center}
\includegraphics[width=.3\linewidth, width=.3\textwidth]{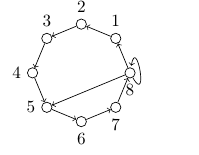}
\end{center}
\caption{} \label{fig:figure3}
\end{figure}

Also the number of matrices with a positive trace and exponent $11$ in $CP_8$ is $\sum\limits_{x=3}^{6}F_{6}(x,3)=12.$ 
\end{ex}
\section{Exponent of $A \in CP_n$ with zero trace}\label{sec:nloop} 
The exponent of the companion matrix of the polynomial $x^n-x-1$ is $\omega_n.$ And this is the only primitive matrix up to isomorphism whose exponent is $\omega_n.$ The following result shows that $E(CP_n) \subset [n, \omega_n]$ and there is only one matrix $A \in CP_n$ with $exp(A)=n.$ We need the following remark to prove the same.
\begin{rem}\label{rem:1}
If $A$ and $B$ are two primitive matrices of order $n$ such that $A \geq B$, then $exp(A) \leq exp(B)$.
\end{rem}
\begin{theorem}\label{unique lower}
Let $A \in CP_{n}$. Then $n\le exp(A)\le 
\omega_n.$  Moreover, there exists a unique $A \in CP_{n}$ such that $exp(A)=n.$
\end{theorem}
\begin{proof}
For $A \in CP_{n}$, $n$ is the smallest length of cycles containing the vertex $1$ in $D(A)$. As a consequence $exp(A)\geq n$.

Let $A$ be the companion matrix of $ f(x)=x^{n}- \sum\limits_{k=1}^{n-1} x^k-1 $.  Then $exp(A:i)=n-i+1$ as $D(A)$ contains a loop at the vertex $n$ and there is an edge from $n$ to each vertex $j$. Hence by Lemma \ref{lemma:exp(A)},  $exp(A)= n.$

Suppose $t \in \{2,3,\ldots,n-2\}$ and $A(t)$ is the companion matrix of $ f(x)=x^{n}- \sum\limits_{k=1,k\neq t}^{n-1} x^k-1 $.  Then $exp(A(t):i)=n-i+2$ as $D(A(t))$ contains a loop at the vertex $n$ and there is an edge from $n$ to each vertex $j,$ $j \neq t.$ Hence $exp(A(t))=n+1.$ Finally, if $A$ is the companion matrix of $ f(x)=x^{n}- \sum\limits_{k=1}^{n-2} x^k-1 ,$ then $exp(A) > n+1$ as $1 \not\xrightarrow{n+1} 1.$ Hence from Remark \ref{rem:1}, there exists a unique matrix $A$ in $CP_{n}$ such that $exp(A)=n.$ 
\end{proof}

The following result can be found in Jia Yu Shao \cite{Shao:2}. Here we are giving a proof for the sake of completeness.
\begin{theorem}\label{two cycles}
Suppose $n \geq 3$ and $A \in CP_n$ such that $D(A)$ contains only two elementary cycles of length $n$ and $s$. Then $exp(A)=n+s(n-2)$. 
\end{theorem}
\begin{proof}
Suppose that $\ell$ is the length of a walk from $1$ to $n-s$ and $\ell \geq n.$ Then $\ell = n+an+bs+(n-s-1)$ for some nonnegative integer $a$ and $b$. It is now known that for all nonnegative integers $i,$ $(n-1)(s-1)+i$ can be expressed as a nonnegative integer linear combination of $n$ and $s.$ But $(n-1)(s-1)-1$ cannot be expressed as a nonnegative integer linear combination of $n$ and $s$. Thus $n+s(n-2)$ is the smallest number which is larger than $n$ and $1 \xrightarrow{(n+s(n-2))^+} n-s.$ Hence $exp(A:1,n-s)=n+s(n-2).$

From \cite{bru}, if $D(A)$ is a primitive digraph  with $n$ vertices and smallest cycle length $s,$ then $exp(D(A)) \leq n+s(n-2).$ Hence the result follows.
\end{proof}

\begin{lemma}\label{lemmm :exp(A:1)}
\begin{enumerate}
Let $n \geq 3,$ $A \in CP_n$ with zero trace and let $D(A)$ be the adjacency digraph of $A$ with the vertex set $V=\{1,2,\ldots,n\}.$ 
\item \label{lem :exp(A:1)}
Then $exp(A)=exp(A:1).$
\item \label{lem:exp(A:1,j)} 
If $j \in V_{1},$ then $exp(A:1,j)=exp(A:1,j-p)+p$, where $p=j-\max\{[1, j] \cap V_2\}$ and $1 \leq p \leq j-1.$
\end{enumerate}
\end{lemma}
\begin{proof}
Proof of part \ref{lem :exp(A:1)}.
Suppose $exp(A:1)=l$, that is, $l$ is the least positive integer such that $1 \xrightarrow{l} j$ for all $j \in V.$ Then $2 \xrightarrow {l-1} j$ for all $j \in V$. Hence $exp(A:2) \leq l-1 < exp(A:1)$. Since $|V|$ is finite, then by a similar argument we can say that $exp(A:n)<exp(A:n-1)<exp(A:n-2)<\dots<exp(A:1)$. Hence the result follows from Lemma \ref{lemma:exp(A)}.

Proof of Part \ref{lem:exp(A:1,j)}.
The existence of $p$ follows from the fact that $1 \in V_{2}$. Since every walk from $1$ to $j$ must contain the vertex $j-p,$ then $1 \xrightarrow{l-p} j-p$ whenever $1 \xrightarrow{l} j.$ Suppose that $exp(A:1,j)=x.$ Then from the definition of $exp(A:1,j)$ we can write $exp(A:1,j-p) \leq x-p.$ But  $exp(A:1,j-p) < x-p$ will contradict the fact that $exp(A:1,j)=x$. Hence  $exp(A:1,j-p)=x-p$ and the result follows.
\end{proof}
Finally, to find the $exp(A)$ where $A \in CP_n$ and the trace of $A$ is zero, it is sufficient to find $exp(A:1,j)$, where $j \in V_{2}$. Suppose $a_1, a_2, \ldots, a_u$ are relatively prime and $F(a_1, a_2, \ldots, a_u)$ denotes the smallest integer such that this integer or any integer larger than this can be expressed as $n_1 a_1+n_2 a_2+\dots+ n_u a_u,$ where $n_r$ is a nonnegative integer for $r=1, 2, \ldots, u.$ The number $F(a_1, a_2, \ldots, a_u)$ is called as the {\it Frobenius number}. This function  $F(a_1, a_2, \ldots, a_u)$ has been discussed by Bateman \cite{bate}, Brauer and Seelbinder \cite{brau:see}, Johnson \cite{john} and Roberts \cite{robert}. It is known that if $m$ and $n$ are relatively prime then $F(m,n)=(m-1)(n-1).$ Roberts has shown that if $a_j=a_0+jd$, $j=0, 1, 2, \ldots, s$, $a_0 \geq 2,$ then $$F(a_0,a_1, a_2, \ldots, a_s)=\left(\bigg\lfloor\frac{a_0-2}{s}\bigg\rfloor+1\right)a_0+(d-1)(a_0-1),$$ where as usual $\lfloor x \rfloor$ denotes the greatest integer $ \leq x.$ The proof of this result has been  simplified by Bateman. Johnson has given an ingenious algorithm which can be used to find $F$ in the case of three variables. It is now easy to establish that for each $j \in V_{2},$  $exp(A:1,j)\leq n+F(\ell_{1},\ell_{2},\ldots,\ell_{z}).$  In particular, we have the following result for $j=1$.
\begin{lemma}\label{lem:exp(A:1,1)}
Let $A \in CP_n$ with zero trace. Then $exp(A:1,1)=n+F(\ell_{1},\ell_{2},\ldots,\ell_{z})$.
\end{lemma}
\begin{proof}
We have $exp(A:1,1) \geq n$. Suppose that $exp(A:1,1)=n+x.$ Clearly, $x \geq 2$ is an integer. It follows from the definition of $exp(A:1,1)$ that $1 \xrightarrow{(n+x)^+} 1$ but $1 \not\xrightarrow{n+x-1} 1.$ Now every walk from $1$ to $1$ is a nonnegative integer linear combination of $\ell_{1},\ell_{2},\dots,\ell_{z}.$ Thus every integer $\geq x$  can be expressed as a nonnegative integer linear combination of $\ell_{1},\ell_{2},\dots,\ell_{z}$ whereas $x-1$ can not be expressed as a nonnegative integer linear combination of $\ell_{1},\ell_{2},\dots,\ell_{z}$. Hence $x=F(\ell_{1},\ell_{2},\dots,\ell_{z})$ and the result follows.
\end{proof}
\begin{cor}
Suppose that $j \in V_2$ and $j\neq 1.$ Then $1 \xrightarrow{l} j$  whenever $1 \xrightarrow{l} 1.$ 
Also $1 \xrightarrow{j-1} j.$ Hence $exp(A:1,j) \leq exp(A:1,1).$
\end{cor}
The following remark evaluates the $exp(A:1,j)$ for some $j \in V_2.$  
\begin{rem}\label{rem:special vertex}
\begin{enumerate}
\item
Suppose $j \in V_{2}$ such that $[j-\ell_1+1, j] \subseteq V_{2}$, then $exp(A:1,j)=n$. Such a vertex $j \in V_{2}$ is named a `${ {\bf special \;\; vertex}}$'.
\item
Suppose $j \in V_{2}$ which is not a special vertex and $j \geq \ell_1.$ Then there exists a $p$ such that $p \in [1, \ell_1-1]$ and $j-p \in V_{1}$. If $\mathcal{M}=max\{p \in [1, \ell_1-1] :j-p \in V_{1}\}$, then $exp(A:1,j) \geq n+\mathcal{M}$. The equality holds if $j-(\mathcal{M}+1) \in V_{2}$ is a special vertex. 
\item
Suppose that $j \in V_{2}$ and $exp(A:1,j)=x,$ where $x$ cannot be expressed as a nonnegative integer linear combination of the elements from $L.$ Then for any $j' \in V_{2}$ with $j' > j$, $exp(A:1,j')=x+(j'-j)$, provided $x+(j'-j) \leq n+F(\ell_{1},\ell_{2}, \dots,\ell_{z})$.

For an example, choose $Y=(0,0,1,1,0,0,0)$; then $exp(A_Y:1,4)=15$ and this implies that $exp(A_Y:1,5)=16$.
\end{enumerate}
\end{rem}
The following remark provides $exp(A)$ for some $A \in CP_n$ with zero trace.
\begin{rem}
\begin{enumerate}\label{rem:max:exp(A)}
\item \label{rem:max1}
Suppose $A \in CP_n$ with zero trace. If  $[2,m(V_1)+1] \subseteq V_{1},$ then $exp(A)=n+F(\ell_{1},\ell_{2}, \dots,\ell_{z})+m(V_1)$.
\item
For $n \geq 3,$ the maximum number of matrices whose exponent can be evaluated by Part \ref{rem:max1} is $\sum\limits_{m=0}^{n-3}T_{m+1}(n-m-3).$ The number of such matrices will be exactly $\sum\limits_{m=0}^{n-3}T_{m+1}(n-m-3)$ if $n$ is a prime number.
\item \label{rem:max2}
Suppose $s, s'$ are two relatively prime integers such that $s > s'.$ If $n \geq F(s,s')$ and $\max\{s-s', s', n-s\}=n-s$, then $exp(A)=2(n-s)+s'(s-1)$ for some $A \in CP_n$ with zero trace. That is, $2(n-s)+s'(s-1) \in E(CP_n).$

For example, it is easy to see from Figure \ref{fig:figure4} that $22 \in E(CP_{10}).$
\begin{figure}[htb]
\begin{center}
\includegraphics[width=.2\linewidth, width=.2\textwidth]{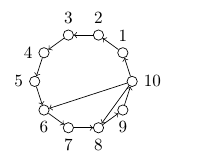}
\end{center}
\caption{} \label{fig:figure4}
\end{figure}
\item \label{rem:max3}
Suppose that $ a_0 \in [\ell_1+1,n-\ell_1+1].$ If $n \geq\left(\lfloor\frac{\ell_1-2}{n-\ell_1-a_0+1}\rfloor+1\right)\ell_1,$ then $exp(A)=n+\left(\lfloor\frac{\ell_1-2}{n-\ell_1-a_0+1}\rfloor+1\right)\ell_1+(a_0-2).$

Figure \ref{fig:figure5} is an example where $\ell_1=4$, $a_0=5,$ $\left(\lfloor\frac{\ell_1-2}{n-\ell_1-a_0+1}\rfloor+1\right)\ell_1=8$ and $21 \in E(CP_{10}).$ 
\begin{figure}[htb]
\begin{center}
\includegraphics[width=.2\linewidth, width=.2\textwidth]{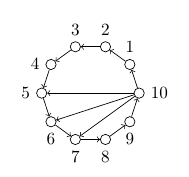}
\end{center}
\caption{} \label{fig:figure5}
\end{figure}
\end{enumerate}
\end{rem}
The following remark provides $exp(A)$ and the numbers in the exponent set $E(CP_n)$ whenever $\ell_1=2.$
\begin{rem}\label{cycle length two}
\begin{enumerate}
\item\label{cycle length two:1} 
Let  $s$ be the smallest odd cycle length in $D(A)$. If $j \in V_{2}$ and $j$ is not a special vertex, then $exp(A:1,j)=n+p-1$, where p is the smallest odd number such that $p < s$ and $j-p \in V_{2}.$  Otherwise, $exp(A:1,j)=n+s-1$.

For an example, choose $Y=(1,0,1,1,0,0,1,0,0,0,1,0,0,1,0).$ Then $exp(A_Y:1,15)=16+3-1$ and $exp(A_Y:1,12)=16+5-1$. 

Thus for an integer $n \geq 4,$ if $A \in CP_n$ with $\ell_1=2$ in $D(A),$ then the exponent of $A$ can be easily evaluated.
\item \label{cycle length two:2} 
Recall that we have $[n, 2n-2] \subset E(CP_n).$ Now if $n$ is odd and $A \in CP_n$ with $\ell_1=2$ in $D(A),$ then $exp(A) \in [n, 3n-4].$  For all nonnegative integers $x \leq \frac{n-3}{2},$ one can see that $2n-1+2x \in E(CP_n)$ by considering the digraph $D(A)$ with $V_1=\{1,2(x+1),n-1\}.$  But if $n$ is even, then the $exp(A)$ cannot go beyond $2n-2.$ The digraph $D(A)$ in Figure \ref{fig:figure6} illustrates the same for odd $n$. Here $n=15$, $x=2$ and $33 \in E(CP_{15}).$ 
\begin{figure}[htb]
\begin{center}
\includegraphics[width=.3\linewidth, width=.3\textwidth]{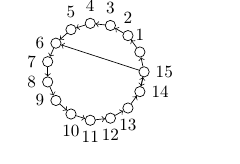}
\end{center}
\caption{} \label{fig:figure6}
\end{figure}
\end{enumerate}
\end{rem}

\section{Problems}\label{sec:prob}

\

In this section, we suggest a few problems
arising naturally from the ideas in the preceding sections.

\begin{enumerate}
\item \label{OoProb1} Here we established the existence of certain numbers in $E(CP_n)$ but we are unable to characterize $E(CP_n)$ completely. Complete characterization is
an interesting problem but it may be quite difficult as there is no known formula for $F(a_0,a_1, \ldots ,a_n)$ when $n \geq 3.$  Consequently, finding $E(CP_n)$ in another approach may be helpful in finding $F(a_0,a_1, \ldots ,a_n)$. Thus further research can be focused here.

\item
We found $exp(A)$ for certain cases such as $A \in CP_n$ with a positive trace, with the smallest cycle length $2$ or $n-1$ in $D(A).$ Hence an immediate question that can be raised is about the remaining $A \in CP_n$, in particular: what will be the exponent set of all $A \in CP_n$ with the smallest cycle length $3$ in $D(A)$? And what will be the exponents of all such matrices? Similar questions can be asked for the smallest cycle length $\ell_1,$ $4 \leq \ell_1  \leq n-2.$ Clearly, this problem can be treated as a simplified form of the problem in Part \ref{OoProb1}. 
\end{enumerate}
%%%%%%%%%%%%%%%%%%%%%%%%%%%%%%%%%%%
% \section{Acknowledgements}
% The authors would like to express their gratitude towards Mr. Shashankaditya Upadhyay for helping in the improvement of the language of this article. The authors also would like to thank the referees for their valuable comments, corrections, and suggestions, which led to an improvement of the original paper.
% %%%%%%%%%%%%%%%%%%%%%%%%%%%%%%%%

\end{document}